\documentclass[11pt,reqno]{amsart}

\usepackage[english]{babel}        

\usepackage[utf8]{inputenc} 
\usepackage[T1]{fontenc}
\usepackage{amsmath} 
\usepackage{amssymb} 
\usepackage{amsthm} 
\usepackage{amsfonts}
\usepackage{graphicx,color} 
\usepackage{texdraw} 
\usepackage{hyperref}
\usepackage{fullpage}
\usepackage{lmodern}
\usepackage{comment}
\usepackage{todonotes}

\newtheorem{theorem}{Theorem}[section] 
\newtheorem{proposition}[theorem]{Proposition} 
\newtheorem{corollary}[theorem]{Corollary} 
\newtheorem{lemma}[theorem]{Lemma}

\newtheorem{example}[theorem]{Example} 
\newtheorem{remark}[theorem]{Remark} 
\newtheorem{definition}[theorem]{Definition}

\def\N{\mathbb{N}}

\def\R{\mathbb{R}}

\def\eps{\varepsilon}

\newcommand{\bbR}{\mathbb{R}}
\newcommand{\bbZ}{\mathbb{Z}}
\newcommand{\bbN}{\mathbb{N}}

\newcommand{\calX}{\mathcal{X}}
\newcommand{\calL}{\mathcal{L}}
\newcommand{\calG}{\mathcal{G}}

\def\norm#1{\left\lVert#1\right\rVert}
\def\abs#1{\left|#1\right|}

\title[Quasilinear bistable equations]{Heteroclinic solutions of  singular quasilinear\\ bistable equations} 

\author{Denis Bonheure}
\address{Denis Bonheure 
\newline \indent Département de Mathématiques, Université Libre de Bruxelles,
\newline \indent CP 214, Boulevard du triomphe, B-1050 Bruxelles, Belgium,
\newline \indent and INRIA- team MEPHYSTO.}
\email{Denis.Bonheure@ulb.ac.be}

\author{Isabel Coelho}
\address{Isabel Coelho 
\newline \indent {\'A}rea Departamental de Matem{\'a}tica, Instituto Superior de Engenharia de Lisboa,
\newline \indent Rua Conselheiro Em{\'i}dio Navarro 1, 1950-062 Lisboa, Portugal.}
\email{icoelho@adm.isel.pt}

\author{Manon Nys}
\address{Manon Nys 
\newline \indent Dipartimento di Matematica Giuseppe Peano, Università degli Studi di Torino, 
\newline \indent Via Carlo Alberto, 10 10123 Torino, Italy.}
\email{manonys@gmail.com}

\thanks{D.B. is supported by INRIA - Team MEPHYSTO, MIS F.4508.14 (FNRS), PDR T.1110.14F (FNRS) 
\& ARC AUWB-2012-12/17-ULB1- IAPAS. I.C. gratefully acknowledges the support of Funda\c c\~ao para a Ci\^encia e a Tecnologia, through the scholarship SFRH/BD/61484/2009, during her stay at the Universit\'e Libre de Bruxelles. M.N. is supported by the project ERC Advanced Grant 2013 n. 339958: "Complex Patterns for Strongly Interacting Dynamical Systems - COMPAT". M.N. was supported by the FNRS when this research was initiated.}

\begin{document} 

\begin{abstract} 
In this note we consider the action functional 
\[
\int_{\mathbb{R}\times \omega} \left( 1 - \sqrt{1 - |\nabla u|^{2}} + W(u) \right) \mathrm{d}\bar{x}
\]
where $W$ is a double well potential and $\omega$ is a bounded domain of $\R^{N-1}$. We prove existence, one-dimensionality and uniqueness (up to translation) of a smooth minimizing phase transition between the two stable states $u=1$ and $u=-1$.
The question of existence of at least one minimal heteroclinic connection for the non autonomous model
\[
\int_{\mathbb{R}} \left( 1 - \sqrt{1 - |u'|^{2}} + a(t)W(u) \right) \mathrm{d}t
\]
is also addressed. For this functional, we look for the possible assumptions on $a(t)$ ensuring the existence of a minimizer. 
\end{abstract}


\maketitle


\section{Introduction}


It is known that any minimizer of the Allen-Cahn energy functional
\[
\mathcal E(u) = \int_{\mathbb{R}\times \omega} \left( \frac12|\nabla u|^{2} + \frac{1}{4}(u^{2}-1)^{2} \right) \mathrm{d}\bar{x}
\]
among functions of $H^{1}_{\text{loc}}(\mathbb{R}\times \omega)$ satisfying $\lim_{x \rightarrow \pm \infty} u(x,y) = \pm 1$ uniformly in $y \in \omega\subset\R^{N-1}$, $\omega$ being a bounded domain of $\R^{N-1}$  for $N\ge1$,  must be a function of the first variable only, and is therefore given by $\tanh(\frac{x+a}{\sqrt2})$, where $a\in \R$, see for instance \cite{KawohlBook,Carbou1995,Alberti2000}.

This minimization problem takes its origin in the van der Waals-Cahn-Hilliard gradient theory of phase separations \cite{CahnHilliard1958,AllenCahn1979,Rowlinson1979}. Assume a two-phase fluid whose density at point $\bar{x}$ is denoted by $u(\bar{x})$ has energy density given by $\frac{1}{4}(u^{2}-1)^{2}$. The densities $u =\pm 1$ correspond to stable fluid phases. When the two distinct phases coexist, they minimize their interaction. At a microscopic scale, this amounts to study the above energy functional. The term $\int\frac{1}{4}(u^{2}-1)^{2}$ penalizes the interaction of the two phases while the presence of the kinetic term $\int|\nabla u|^{2}$ accounts to penalize the formation of unnecessary interfaces. In the context of superconductivity, the functional $\mathcal E$ is known as the Ginzburg-Landau energy without magnetic field, see e.g. \cite{BethuelBrezisHeleinBook}.

\bigbreak

In this paper, we first consider the modified autonomous free-energy functional 
\begin{equation} 
		\label{functionalJ}
\mathcal{J}(u) = \int_{\mathbb{R}\times \omega} \left( 1 - \sqrt{1 - |\nabla u|^{2}} + W(u) \right) \mathrm{d}\bar{x}.
\end{equation}

We suppose that  $N\ge1$, $\omega\subset\R^{N-1}$ is a bounded domain and $W$ is a double well potential by which we mean that 
\begin{itemize}
\item[$(W_1)$] $W \in C^{1}(\mathbb{R})$,	
\item[$(W_2)$] $W(-1) = W(1) = 0$ and $W(s)>0$ if $s\ne \pm 1$.
\end{itemize}
The Allen-Cahn potential $W(s)=\frac{1}{4}(s^{2}-1)^{2}$ is of course a typical example, though we will not assume symmetry nor non degeneracy of the bottoms of the wells. 

We look for a minimal transition from the equilibrium state $u=-1$ to the equilibrium state $u=+1$ by minimizing $\mathcal{J} : \mathcal{X} \rightarrow \mathbb{R}^{+} \cup \{+\infty \}$, where
\begin{multline}  \label{spaceX}
\mathcal{X}:= \Big\{u \in  W^{1,\infty}(\mathbb{R}\times \omega)	 \mid 	\norm{\nabla u}_{L^\infty} \le 1	
\text{ and } \lim_{x \rightarrow \pm \infty} u(x,y) = \pm 1 \mbox{ uniformly in y} \in \omega \Big\}.
\end{multline}
Since  $1- \sqrt{1 - |\nabla u|^{2}}$ is an uniformly strictly convex function of the gradient, the volume integral
\begin{equation}	\label{eq:vol-int}
K(u) = \int_{\mathbb{R}\times \omega} \left( 1 - \sqrt{1 - |\nabla u|^{2}} \right) \mathrm{d}\bar{x}
\end{equation}
can be seen by analogy with classical models as an energy which attempts to minimize the area of the interface between the two phases. Such a model supposes that all transitions occur with an a priori bounded gradient (we have normalized this a priori bound to $1$ but this is clearly not restrictive). Due to the singularity at this threshold value of the norm of the gradient, sharp transitions between the two phases are clearly penalized. On the other hand, we have the pointwise  estimates
\begin{equation}
		\label{quadraticEstimates}
\frac12 |\nabla u|^2 \le 1-\sqrt{1-|\nabla u|^2} \le |\nabla u|^2.
\end{equation}
Therefore, for small values of the gradient, the energy behaves quadratically while it reaches its extremal value $1$ with infinite slope when $|\nabla u|$ reaches its upper bound. As a consequence, the diffusion flux $\frac{\nabla u}{\sqrt{1 - |\nabla u|^{2}}}$ is singular at the maximal value of the gradient. 

\medbreak

The quasilinear operator 
\[
Q(u) = -\nabla. \left(\frac{\nabla u}{\sqrt{1 - |\nabla u|^{2}}} \right),
\]
which appears, at least formally, as the derivative of the volume integral \eqref{eq:vol-int},
has been extensively studied in the recent years. We quote for example \cite{BereanuMawhin2004,BereanuMawhin2007,BereanuMawhin2008} which are concerned with a one-dimensional Neumann problem. For the Dirichlet problem we mention \cite{BereanuMawhin2007,BereanuMawhin2008,CoelhoCorsatoObersnelOmari2012,BereanuPetruTorres2013,
BereanyPetruTorres2013-2,delaFuenteRomeroTorres2015}, in one-dimensional or radial settings, and \cite{CorsatoObersnelOmariRivetti2013,BereanuJebeleanSerban2015,BereanuJebeleanMawhin2016corr} in a general bounded domain; 
while the periodic problem is addressed in \cite{BereanuMawhin2007,BereanuJebeleanMawhin2010,MawhinBrezis2010,BrezisMawhin2011,
JebeleanMawhinSerban2015,JebeleanMawhinSerban2016}. 
We also refer to the survey of Mawhin \cite{Mawhin2014} and the references therein to be as exhaustive as possible. 
To our knowledge, less attention has been given to problems in unbounded domains or in the whole space. We mention in that direction \cite{BonheureDerletDeCoster2012,Azzollini2014,Azzollini2015,BonheureDaveniaPomponio2015,
BereanudelaFuenteRomeroTorres2016,
BonheureCoelhoDeCoster2016} for ground state type solutions in the whole space and \cite{CupiniMarcelliPapalini2011,CoelhoSanchez2014} where the authors considered heteroclinic solutions in  frameworks that include the operator $Q(u)$. 
The operator $Q$ is a classical object in Riemannian geometry. Indeed, if $u$ is a smooth function of $N$ variables with $\|\nabla u\|_{L^{\infty}}<1$, the graph of $u$ is (a piece of) a $N$-dimensional surface in the Lorentz-Minkowski space $\mathbb L^{N+1}$ whose local mean curvature is given by $Q(u)$, see for instance \cite{BartnikSimon1983}. The volume integral $K(u)$   can then be seen as the area integral in $\mathbb L^{N+1}$ and surfaces of minimal or maximal area solve the equation $Q(u)=0$.  Within a quite different context, namely in the nonlinear theory of electromagnetism, this operator appears in the formulation of the Born-Infeld field theory, see \cite{BI,Yang2000,Kiessling2012,BonheureDaveniaPomponio2015} and the references therein. 
 
\bigbreak

In the first part of this paper, we prove a Gibbons-type conjecture in cylinders for the singular operator $Q(u)$. Namely, we prove that the free-energy functional $\mathcal{J}$ has a minimizer in $\mathcal{X}$, and that this minimizer has to be one-dimensional, i.e. independent from $y \in \omega$.
\begin{theorem} 
		\label{theoremfinal}
Assume $(W_1)$ and $(W_2)$. The functional $\mathcal{J}$ defined in \eqref{functionalJ} attains its infimum in $\calX$. Moreover any minimizer of $\mathcal{J}$ depends only on the first variable $x \in \mathbb{R}$ and is the unique solution, up to translations, of the equation
\begin{equation} \label{eqNdim}
\displaystyle \left( \frac{u'}{\sqrt{1 - |u'|^2}}\right)' = W^{\prime}(u) \quad\hbox{ in }\bbR,		
\end{equation}
with the boundary conditions
\begin{equation} 
			\label{boundarycondNdim}
\lim_{x\to\pm\infty}\big(u(x),u'(x)\big)=(\pm 1,0).			
\end{equation}
In addition,  $\norm{ u^{\prime} }_{L^\infty}  < 1$, $u\in C^\infty$ and  $u$ satisfies the energy conservation identity
\begin{equation} 	
			\label{cons.energy}
1-\frac{1}{\sqrt{1- |u^\prime|^2}} + W(u) = 0 .	
\end{equation}
\end{theorem}

The main technical difficulty in describing the phase transition consists in showing that the minimizers of the energy functional satisfy the (formaly) associated Euler-Lagrange equation. This leads to prove that the gradient of the minimizer, or simply the derivative once we know the minimizer depends on a single space variable, is bounded away from $1$.

\bigbreak

In the second part of the paper, we focus on a time-dependent one-dimensional system. In special relativity, see for instance \cite{Erkal2000}, the Lagrangian
\[
L_a(t,u(t),u^\prime(t)) = mc^{2}\left( 1 - \sqrt{1 - \frac{u^{\prime \, 2}(t)}{c^{2}}}\right) + a(t)W(u(t))
\]
describes the motion of a particle of rest mass $m$ whose time-dependent potential energy is given by $-a(t)W$. In the non-relativistic limit $c\to+\infty$, this Lagrangian yields the usual equation of motion
\[
m u''(t) = a(t)W'(u(t)).
\]
In the relativistic regime, assuming that the motion is strictly timelike, that is $|u'(t)| < c$, we can use the phase-space variables 
\[
(u,p)= \left(u,\frac{m u'}{\sqrt{1-\frac{{u'}^{2}}{c^{2}}}} \right)
\] 
so that the equations of motion can be written as a regular system of first order equations
\begin{equation*} 
u' = \frac{p}{\sqrt{m^2+\frac{p^{2}}{c^{2}}}}, \qquad \text{ and } \qquad  p'  = a(t)W'(u) .
\end{equation*}

\medbreak

For simplicity, in this paper we fix $m=1$ and $c=1$, and consider the action (observe that we emphasize the dependence of the functional on the weight $a(t)$)
\begin{equation} 
		\label{functionalL}
\mathcal{L}_a(u) := \int_{- \infty}^{+ \infty} \left(  1 - \sqrt{1 - |u'|^{2}} + a(t)W(u) \right) \mathrm{d}t.
\end{equation}

As in the first part, we will look for minimal heteroclinic solutions connecting the extremal equilibria, namely asymptotic motions $u(t)$ minimizing the action functional $\mathcal{L}_a$ in the space $\mathcal{X}_1$ defined as
\begin{equation} 
		\label{spaceX1D}
\mathcal{X}_1:= \Big\{ u \in  W^{1,\infty}(\mathbb{R}) \mid   \norm{ u^{\prime} }_{L^\infty}  \le 1 \text{ and } \lim_{t \rightarrow \pm \infty} u(t) = \pm 1  \Big\} .
\end{equation}
First, we will assume that $a(t)$ satisfies
\begin{itemize}
\item[$(a_1)$] there exist $a_1, a_2 \in\bbR$ such that  $0 < a_1 \leq a(t) \leq a_2$ for all $t \in \mathbb{R}$.
\end{itemize}
	
\medbreak

The main point here is to provide conditions on the non-autonomous term $a(t)$ which secure the existence of a minimal timelike heteroclinic solution. Observe that when $a(t)$ is constant,  the existence of a minimizer of $\mathcal{L}_a$ is already provided by Theorem \ref{theoremfinal}. 
Our main results are the two following theorems.

\begin{theorem} 	\label{thm.structural}
Assume $(W_1)$, $(W_2)$ and $(a_1)$. If
\begin{itemize}
\setlength{\itemsep}{5pt}
\item[$(b_1)$] there exists $b\in L^\infty(\bbR)$ such that $a(t) \le b(t)$ for every $t\in\R$, 
\begin{equation*}
\quad \lim_{\abs{t}\to+\infty} \big(b(t)-a(t)\big)=0
\end{equation*} 
and $\displaystyle \inf_{\calX_1} \calL_b$ is attained,  
\end{itemize}
then $ \calL_a$ attains its infimum at $u_a\in\calX_1$. Moreover, $u_a$ is a $C^1 \cap W^{2,2}_{\text{loc}}$ solution of 
\begin{align} \label{eq:nonAutonomous}
\left( \frac{u'}{\sqrt{1 - |u'|^2}}\right)'  & =   a(t) W^{\prime}(u) \quad\text{ in }\bbR, 
		\\ 	 \label{eq:nonAutonomous2}
\lim_{t \to \pm \infty}  \big(u(t), u^{\prime}(t) \big)  & =  (\pm 1 ,0),					
\end{align}
and $\| u_a \|_{L^\infty} < 1$.
\end{theorem}

\begin{theorem}[Periodic weight]	\label{thm.periodic}
Assume $(W_1)$, $(W_2)$ and $(a_1)$.
Suppose  also that $a$ is $T$-periodic for some $T>0$. Then $\calL_a$ attains its infimum at $u_a \in \calX_1$. Moreover, $u_a$ is a heteroclinic $C^1 \cap W^{2,2}_{\text{loc}}$ solution of \eqref{eq:nonAutonomous}--\eqref{eq:nonAutonomous2} and $\| u_a \|_{L^\infty} < 1$.
\end{theorem}
 
By combining Theorem \ref{thm.structural} with Theorem \ref{theoremfinal} and Theorem \ref{thm.periodic}, we deduce respectively the two following corollaries.

\begin{corollary}[Asymptotically constant  weight]
Assume $(W_1)$, $(W_2)$ and $(a_1)$.
If in addition
\begin{equation*}
\lim_{\abs{t}\to +\infty} a(t)=a_2
\end{equation*}
then $ \calL_a$ attains its infimum at some $u_a \in \calX_1$. Moreover, $u_a$ is a heteroclinic $C^1 \cap W^{2,2}_{\text{loc}}$ solution of \eqref{eq:nonAutonomous}--\eqref{eq:nonAutonomous2} and $\| u_a \|_{L^\infty} < 1$.
\end{corollary}

\begin{corollary}[Asymptotically periodic weight]
Assume $(W_1)$, $(W_2)$ and $(a_1)$.
In addition,  suppose that there exists $b\in L^\infty(\bbR)$, $T$-periodic for some $T>0$, such that $(b_1)$ holds. Then $ \calL_a$ attains its infimum at $u_a \in \calX_1$. Moreover, $u_a$ is a heteroclinic $C^1 \cap W^{2,2}_{\text{loc}}$ solution of \eqref{eq:nonAutonomous}--\eqref{eq:nonAutonomous2} and $\| u_a \|_{L^\infty} < 1$.
\end{corollary}

\bigbreak

Finally, we conclude this work discussing the case where the functional $\calL_a$ satisfies additional  symmetry properties. We will consider  $(W_1)$ and the following additional assumptions
\begin{itemize}
\item[($W_{2}^\prime$)] $W(s) > 0$ if $s \in {(-1,1)}$ and $W(s) = 0$ for all $s \in \bbR \setminus (-1,1)$,
\item[$(W_3)$] $W(s) = W(-s)$ for all $s \in \mathbb{R}$,
\item[$(a_1^\prime)$] $a \in L^\infty_{\text{loc}}(\mathbb{R})$,
\item[$(a_2)$] $a(s) \le a(t)$ for a.e.~$0 \le s \le t$,
\item[$(a_3)$] $a(t) = a(-t)$ for a.e.~$t \in \mathbb{R}$,	
\item[$(a_4)$] 	there exists   $T>0$   such that   $a(t) >0$   for all   $t >T$.		
\end{itemize}

Notice that we do not impose $(a_1)$ in this setting. Moreover, $a(t)$ can take negative values on a compact set. For this reason, we consider the new assumption $(W_2')$ that forces the solution to take values in $[-1,1]$, which would not be clear (and maybe not true) under $(W_2)$. Since $W$ and $a$ are even, we will look for an odd minimizer of  $\calL_a$ in the set
\[
\calX_1^{odd} :=\{  u \in \calX_1 \mid  u(-t) = -u(t)  \  \text{ for all } t\in \bbR  \} .
\]
Obviously, this does not exclude the existence of other critical points of $\calL_a$ in $\calX_1$ that are not antisymmetric. In particular it does not imply that the global minimizer of $\calL_a$ in $\calX_1$ is an antisymmetric function. This is an open interesting problem even for simpler semilinear or quasilinear operators.

\medbreak

Due to the symmetry assumptions, we can write
\begin{equation*}
\calL_a(u)  =    2 \int_{0}^{+\infty}  \left(  1 - \sqrt{1 - | u^{\prime} |^{2}} + a(t) W(u) \right)  \, \mathrm{d}t	
\end{equation*}
whenever $u\in \calX_1^{odd} $. Therefore, we will study the functional
\begin{equation*}  
\mathcal{L}_{a}^A(u)   =  	  \int_{0}^{+\infty}  \left( 1 - \sqrt{1 - | u^{\prime} |^{2}}  + a(t) W(u) \right) \, \mathrm{d}t		
\end{equation*}
in the space
\begin{equation*} 
\mathcal{X}_{A} = \big\{ u \in W^{1,\infty} \big( [0, \infty) \big) \mid  u(0) = 0  , \  \norm{ u^{\prime} }_{L^\infty} \leq 1  \  \text{ and } \lim_{t \to +\infty} u(t) = 1  \big\}  .
\end{equation*}

In this setting, the possible loss of compactness is overcome by the monotonicity of the weight $a$. Our result for this case is the following theorem.
\begin{theorem} \label{thm:AntiSym}
Assume $(W_1)$, $(W_2')$, $(W_3)$, $(a_1^\prime)$, $(a_2)$, $(a_3)$ and $(a_4)$. The functional $\mathcal{L}_{a}^A$ attains its infimum at some $u_a \in \mathcal{X}_{A}$. Moreover, $u_a$ is a $C^1 \cap W^{2,2}_{\text{loc}}$ solution of \eqref{eq:nonAutonomous}--\eqref{eq:nonAutonomous2} and $\| u_a \|_{L^\infty} < 1$.
\end{theorem}

\bigbreak

The paper is organized as follows. In Section~\ref{sec:autonomous}, we prove Theorem \ref{theoremfinal}. The proof of this theorem will be divided in several steps. First, we basically reduce the problem to a one-dimensional setting. To do this, we use the classical monotone rearrangement used in \cite{KawohlBook,Carbou1995,Alberti2000} to prove that any minimizing sequence of $\mathcal{J}$ in $\mathcal X$ can be assumed to depend only on the first variable. The arguments are quite classical, though some technical adaptations to our functional setting are required.
The existence of a one-dimensional minimizer is somehow standard. Then, using a stretching argument on the minimizer (see Definition \ref{utheta}), we obtain an a priori bound on the derivative, and therefore conclude that the minimizer solves the Euler-Lagrange equations. Finally, an additional argument is used to prove that any minimizer is one-dimensional, since we are not able to use the rearrangement argument for functions $u$ with $\|\nabla u\|_{L^{\infty}}=1$. However, since the derivative of the volume integral $K(u)$ has a convex singularity on these functions, it is expected that such minimizer cannot exist. We confirm this by performing a regularization of the volume integral, inspired from \cite{CoelhoCorsatoObersnelOmari2012}.

In Section \ref{section:nonAutonomous}, we consider the non-autonomous functional $\mathcal L_a$ given by \eqref{functionalL}
 and look for minimizers of $\calL_a$ in the functional space $\calX_1$.
There are some differences with respect  to the previous case, in particular, the conservation of energy is no longer valid.
We will therefore  argue differently to prove the  a priori bound on the derivative of the minimizers.
Then, we discuss in  Theorems \ref{thm.structural} and \ref{thm.periodic} conditions on the weight $a(t)$ that guarantee the existence of minimizers of $\calL_a$.
We prove these theorems reasoning in the same way as in the proofs of respectively \cite[Theorem 2.5, Theorem 2.4]{BonheureObersnelOmari2013}. 
Finally, Subsection \ref{subsection:oddHeteroclinics} is devoted to the study of a functional with additional symmetry properties. In particular, we prove Theorem \ref{thm:AntiSym}. 
In this case,  the monotonicity of the weight $a(t)$  plays an important role in excluding possible losses of compactness.

\subsection*{Acknowledgment}
The authors are grateful to B. Kawohl and S. Kr\"{o}mer for a fruitful discussion that lead to Theorem \ref{thm:nouveau1D}. 


\section{Optimal shape of phase separation in cylinders} \label{sec:autonomous}



\subsection{Monotone rearrangement}


To prove Theorem \ref{theoremfinal}, we will use a notion of monotone rearrangement of functions that we first recall. This rearrangement was first introduced in \cite{Carbou1995}, see also \cite{Farina1999,KawohlBook} and \cite{Alberti2000} for a deeper study. This rearrangement is a generalization of the classical monotone rearrangement for functions of a single variable. We believe that the adaptation of this rearrangement to our setting has an interest in itself. We anticipate that a shorter path to get one-dimensionality will be provided in Theorem \ref{thm:nouveau1D}. However, the monotone rearrangement has the advantage of directly providing a monotone minimizing sequence. 

\medbreak

For any function $u : \R\times \omega\to [-1,1]$, we define the level sets of $u$ as
\begin{align*}
\Omega_{c} = \left\{ \begin{aligned}
& \left\{ (x,y) \in \mathbb{R}\times \omega \mid 0 \leq u(x,y) \leq c \right\}  \quad & \text{ if } 0\le c \le 1, \\
& \left\{ (x,y) \in \mathbb{R}\times \omega \mid c < u(x,y) < 0 \right\} \quad &  \text{ if } -1 \le c < 0.
\end{aligned} \right.
\end{align*} 
\medbreak
Assume that $u : \R\times \omega\to [-1,1]$ belongs to $\mathcal X$, defined in \eqref{spaceX}, and observe that it implies that its level sets $\Omega_{c}$ have finite measure for any $c\in(-1,1)$, 
while  $m_{N}(\Omega_{-1}) = m_{N}(\Omega_{1}) = +\infty$, where $m_{N}$ represents the Lebesgue measure in dimension $N$.
Define the associated distribution function $a_{u} : (-1,1)\to\R$ by
\begin{align*}
 a_{u}(c) = \left\{ \begin{aligned} & \frac{m_{N}(\Omega_{c})}{m_{N-1}(\omega)} \quad  \text{ if }  0 \le c < 1 \\
                                   - & \frac{m_{N}(\Omega_{c})}{m_{N-1}(\omega)} \quad  \text{ if } -1 < c < 0,
                                   \end{aligned} \right.
\end{align*}
where we may assume  that $m_{N-1}(\omega)=1$, if $N=1$.
In general, the distribution function is nondecreasing, but since functions $u \in \mathcal{X}$ are continuous, the function $a_u$ is in fact increasing. In addition, one easily checks that it is right-continuous by construction.  
We emphasize that $a_{u}$  can have  at most a countable set  of jump discontinuities,
corresponding to the values of   $c \in (-1,1)$  for which  $m_{N}( \{(x,y) \in \mathbb{R}\times \omega \, |\,  u(x,y) = c\}) > 0$.

Next, we  define the one-dimensional nondecreasing rearrangement $u^{\star} : \mathbb{R}\times \omega \rightarrow [-1,1]$ of $u$ through its level sets
\begin{align*}
\Omega^{\star}_{c} = \left\{ \begin{aligned}
& \left\{ (x,y) \in \mathbb{R}\times \omega \mid 0 \leq u^{\star}(x,y) \leq c \right\} = [0, a_{u}(c)]\times \omega  & \text{ if } 0\le c < 1, \\
& \left\{ (x,y) \in \mathbb{R}\times \omega \mid c < u^{\star}(x,y) < 0 \right\} = (a_{u}(c), 0)\times \omega &  \text{ if } -1 < c < 0,
\end{aligned} \right.
\end{align*}
and	
\[
\Omega^{\star}_{1}  =  {[0, + \infty)} \times \omega \quad \text{ and }	\quad \Omega^{\star}_{-1}  =  {(-\infty, 0)}  \times \omega .
\]
We immediately remark that the function $u^{\star}$ depends only on the $x$-variable because its level sets are cylinders. In the sequel, we use the notation $u^{\star}(x) = u^{\star}(x,y)$	for all $y\in\omega$.

Basically, we aim to define $u^{\star}$ as the inverse of the function $a_{u}(c)$. This will be indeed the case if $a_{u}$ is continuous. More precisely, denoting by $I\subset\R$ the range of $a_{u}$, $x_{i}=\inf I$, and $x_{s}=\sup I$, we first set $u^{\star}(x) = -1$ for any $x\le x_{i}$ if $x_{i} > -\infty$ and $u^{\star}(x) = 1$ for any $x\ge x_{s}$ if $x_{s} < +\infty$. Next, if $x \in I$, there exists $c \in (-1,1)$ such that $a_{u}(c) = x$ and we simply define $u^{\star}(x) = c$. For any $x\in (x_{i},x_{s})\setminus I$, we consider 
\[
\xi_{x} = \inf\{\xi \geq x \, |\,  \xi \in I \}.
\]  
The value $\xi_{x} $ belongs to $I$ because $a_{u}$ is right-continuous and increasing. Therefore there exists $c_{x} \in (-1,1)$ such that $a_{u}(c_{x}) = \xi_{x}$ and we define $u^{\star}(x)= c_{x}$. One  easily sees that $u^{\star}$ is a nondecreasing function with $u^{\star}(0) = 0$. We observe that $u^\star$ is an increasing function only if $a_u$ is continuous. One can  check from its definition  that $u^{\star}$ is continuous. This will also be a consequence of the next lemma which shows that starting from a function $u\in \mathcal X$, we have $u^{\star}\in \mathcal X$.

\begin{lemma} \label{lemmalipschitz}
If $u\in\mathcal X$, takes values in $[-1,1]$ and $|\nabla u|\in L^{2}(\mathbb{R} \times \omega)$, then $u^{\star}\in\mathcal X$ and
\begin{equation} 
		\label{eq:ineq-Lip}
 \| \nabla u^{\star} \|_{L^{\infty}(\mathbb{R} \times \omega;\R^{N})} \leq \| \nabla u \|_{L^{\infty}(\mathbb{R} \times \omega;\R^{N})}.
\end{equation}
\end{lemma}

This result is a Pólya-Szegö type inequality. In order to prove this statement, we will use the following (standard) lemma whose proof is included for completeness. 

\begin{lemma}  \label{lemmaLinfini}
If $f\in L^{p}(\mathbb{R} \times \omega)$ for all $1 \leq p < +\infty$ and $\sup_{p} \| f\|_{L^{p}(\mathbb{R} \times \omega)} < +\infty$, then $f\in L^{\infty}(\mathbb{R} \times \omega)$.
\end{lemma}

\begin{proof}
Assume by contradiction that $\| f \|_{L^{\infty}(\mathbb{R} \times \omega)} = +\infty$. Then for all $R > 0$, there exists a set $\Omega_{R} \subset \mathbb{R} \times \omega$ with positive Lebesgue measure such that $|f(x)| \geq R$ in $\Omega_{R}$. This implies that
\begin{align*}
\| f \| _{L^{p}(\mathbb{R} \times \omega)} \geq \left( \int_{\Omega_{R}} |f(x)|^{p} \right)^{1/p} \geq m_{N}(\Omega_{R})^{1/p} R.
\end{align*}
Taking the limit as $p \rightarrow + \infty$, we obtain 
\[
\liminf_{p \rightarrow +\infty} \| f \| _{L^{p}(\mathbb{R} \times \omega)} \geq R.
\]
Since $R>0$ is arbitrary, this contradicts the uniform boundedness of $f$ in the $L^{p}$-norm.
\end{proof}

We will also use the following consequence of the P\'{o}lya-Szeg\"o type inequality proved by Alberti \cite[Theorem 2.10]{Alberti2000}.

\begin{theorem} \label{theoremAlberti}
Assume $g : [0, + \infty) \rightarrow [0, + \infty)$ is convex, increasing and $g(0) = 0$. Then, for every $u \in \mathcal{X}$  {taking values in $[-1,1]$}, we have
\begin{equation*}
 \int_{\mathbb{R}\times \omega} g(|\nabla u^{\star}|)  \leq   \int_{\mathbb{R}\times \omega} g(|\nabla u |) .
\end{equation*}
Moreover, when the left-hand side is finite, equality holds if and only if there exists $a\in\R$ such that $u (x+a,y)= u^{\star}(x)$ for every $(x,y)\in\R\times\omega$.
\end{theorem}

The proof of Lemma \ref{lemmalipschitz} is now a consequence of Theorem \ref{theoremAlberti} and Lemma \ref{lemmaLinfini}.

\begin{proof}[Proof of Lemma \ref{lemmalipschitz}]

First, observe that since $\nabla u : \mathbb{R} \times \omega \rightarrow \mathbb{R}^{N}$ is square integrable and $|\nabla u|$ is uniformly bounded by $1$, we have
\begin{align*}
 \| \nabla u \|_{L^{p}(\mathbb{R} \times \omega; \mathbb{R}^N)}  =  \left( \int_{\mathbb{R} \times \omega} | \nabla u |^{p} \right)^{1/p} = \left( \int_{\mathbb{R} \times \omega} | \nabla u |^{2} | \nabla u |^{p-2} \right)^{1/p} \leq \left( \int_{\mathbb{R} \times \omega} |\nabla u |^{2} \right)^{1/p} \leq M,
\end{align*}
where $M$ is independent of $p$. This  shows that $| \nabla u |\in L^{p}(\mathbb{R} \times \omega)$ for all $2  \leq p < +\infty$ and also that $\sup_{p \geq 2} \| \nabla u \|_{L^{p}(\mathbb{R} \times \omega;\R^{N})} < +\infty$. Using Theorem \ref{theoremAlberti} with $g(t) = t^{p}$, for $p\ge 2$, we infer that 
\begin{align*}
\| \nabla u^{\star} \|_{L^{p}(\mathbb{R} \times \omega;\R^{N})} \leq \| \nabla u \|_{L^{p}(\mathbb{R} \times \omega;\R^{N})}
\end{align*}
which implies that $|\nabla u^{\star}|\in L^{p}(\mathbb{R} \times \omega)$ for all $2 \leq p < +\infty$ and gives an uniform bound as well. Lemma \ref{lemmaLinfini} allows to conclude that $|\nabla u^{\star}|\in L^{\infty}(\mathbb{R} \times \omega)$. To prove that $u^{\star}\in \mathcal X$, we first claim that
\begin{equation*}
\limsup_{p \rightarrow +\infty} \| \nabla u^{\star} \|_{L^{p}(\mathbb{R} \times \omega;\R^{N})} \leq \| \nabla u \|_{L^{\infty}(\mathbb{R} \times \omega;\R^{N})}.
\end{equation*}
Indeed, using again Theorem \ref{theoremAlberti} with $g(t) = t^{p}$, we deduce that 
\begin{align*}
\| \nabla u^{\star} \|_{L^{p}(\mathbb{R} \times \omega;\R^{N})} \leq \| \nabla u \|_{L^{p}(\mathbb{R} \times \omega;\R^{N})} \leq \| \nabla u \|_{L^{\infty}(\mathbb{R} \times \omega;\R^{N})}^{\frac{p-2}{p}} \left( \int_{\mathbb{R} \times \omega} | \nabla u |^{2} \right)^{1/p},
\end{align*}
so that our claim easily follows by taking the limit as $p \rightarrow +\infty$ in both sides of the inequality.

To establish \eqref{eq:ineq-Lip}, it is enough to prove that
\begin{equation*} 
 \| \nabla u^{\star} \|_{L^{\infty}(\mathbb{R} \times \omega;\R^{N})}	\leq	\liminf_{p \rightarrow +\infty} \| \nabla u^{\star} \|_{L^{p}(\mathbb{R} \times \omega;\R^{N})} .
\end{equation*}
Fix $\varepsilon > 0$ and define $\Omega_{\varepsilon} \subset \mathbb{R} \times \omega$ with finite positive measure such that
\begin{align*}
|\nabla u^{\star}(x)| \geq (1 - \varepsilon) \| \nabla u^{\star} \|_{L^{\infty}(\mathbb{R} \times \omega;\R^{N})} \quad \mbox{ for all } x \in \Omega_{\varepsilon}.
\end{align*}
Then
\begin{align*}
\| \nabla u^{\star} \|_{L^{p}(\mathbb{R} \times \omega;\R^{N})} & \geq (1 - \varepsilon) \left( \int_{\Omega_{\varepsilon}} \| \nabla u^{\star} \|_{L^{\infty}(\mathbb{R} \times \omega;\R^{N})}^{p}\right)^{1/p} = (1 - \varepsilon) m_{N}(\Omega_{\varepsilon})^{1/p} \| \nabla u^{\star} \|_{L^{\infty}(\mathbb{R} \times \omega;\R^{N})}.
\end{align*}
Taking the limit as $p \rightarrow +\infty$, we obtain
\begin{equation*}
\liminf_{p \rightarrow +\infty} \| \nabla u^{\star} \|_{L^{p}(\mathbb{R} \times \omega;\R^{N})}  \geq  (1 - \varepsilon) \| \nabla u^{\star} \|_{L^{\infty}(\mathbb{R} \times \omega;\R^{N})} 
\end{equation*}
and since $\varepsilon$ is arbitrary, the conclusion follows.  

As an immediate consequence of \eqref{eq:ineq-Lip}, we deduce  that $|\nabla u^{\star}|$ is uniformly bounded by $1$
and therefore $u^\star \in \calX$,  as $\lim_{x \to \pm \infty} u^\star(x) = \pm1$  by construction.
\end{proof}

\begin{remark}
Although  the uniform square-integrability of $|\nabla u|$ in Lemma~\ref{lemmalipschitz} seams a restrictive assumption at first, it will be naturally satisfied by the elements of any minimizing sequence for the functional $\mathcal{J}$ as it follows from the boundedness of the energy and inequality \eqref{quadraticEstimates}.
\end{remark}

In the following theorem, we  prove that the volume integral  $K(u)$, defined in \eqref{eq:vol-int},  does not increase under the monotone rearrangement.

\begin{theorem} \label{theoremyoung}
For any $u \in \mathcal{X}$, {taking values in $[-1,1]$ and} such that $|\nabla u|\in L^{2}(\mathbb{R} \times \omega)$, we have
\begin{equation} \label{desyoung}
  \int_{\mathbb{R}\times \omega} \left( 1 - \sqrt{1 - |\nabla u^{\star} |^{2}} \right)  \leq  \int_{\mathbb{R}\times \omega} \left( 1 - \sqrt{1 - |\nabla u |^{2}} \right).
\end{equation}
Moreover, if $\| \nabla u \|_{L^{\infty}(\mathbb{R} \times \omega; \mathbb{R}^N)}<1$, we have equality if and only if there exists $a\in\R$ such that $u (x+a,y)= u^{\star}(x)$ for every $(x,y)\in\R\times\omega$.
\end{theorem}

To prove this theorem, one would like to rely on Theorem \ref{theoremAlberti} with the function $g(t) = 1 - \sqrt{1 - t^{2}}$. However, this function  is not defined in the whole interval $[0, +\infty)$ and cannot be extended to this interval as a convex function, since at $t = 1$, its derivative is not finite. Anyway, we  can overcome this problem arguing by truncation close to $t = 1$.

\begin{definition} \label{gtronqué}
For all $n\in\bbN_0$, we set $h_{n}, \Psi_{n} : [0, + \infty) \rightarrow [0, + \infty)$ such that
\begin{align*}
h_n (t) = \left\{ \begin{aligned}
& 1 - \sqrt{1 - t}  \quad  & \text{if }\ t \leq 1 -  \frac{1}{n^2}, \\
& a_n + b_n ( t - 1 + \frac{1}{n^2} ) + c_n ( t - 1 + \frac{1}{n^2} )^{2}  \quad  & \text{if }\ t > 1 -  \frac{1}{n^2},
\end{aligned} \right. 
\end{align*}
and
\[
\Psi_n (t) = h_n (t^2),
\]
where the coefficients $a_n$, $b_n$ and $c_n$ are chosen in such a way that $\Psi_{n}$ is of class $C^{2}$. In particular, $a_n = 1-\frac{1}{n}$, $b_n = \frac{n}{2}$ and $c_n = \frac{n^3}{8}$.
\end{definition}

\begin{proof}[Proof of Theorem \ref{theoremyoung}]

The auxiliary function $\Psi_{n}$ satisfies the hypothesis in Theorem \ref{theoremAlberti}, therefore the following inequality holds
\begin{equation}\label{ineq:gepsilon}
  \int_{\mathbb{R} \times \omega} \Psi_{n}(|\nabla u^{\star}|)  \leq   \int_{\mathbb{R} \times \omega} \Psi_{n}(|\nabla u |)
\end{equation}
for all $u \in\calX$ taking values in $[-1,1]$ and all $n\in\bbN_0$. The square integrability of $ |\nabla u|$ implies that the set $A:=\{ (x,y) \in \bbR\times\omega \mid \abs{\nabla u(x,y)}^2\ge1/2 \}$ has finite measure.
It also implies that there is a function  $h\in L^1(\bbR\times\omega)$  given by
\begin{equation*}
h (x,y) = 
\left\{ 
\begin{array}{ll} 
1 \quad & \text{if }\ (x,y)\in A, \\
\abs{\nabla u(x,y)}^2  \quad & \text{if }\  (x,y)\not\in A,
\end{array}
\right. 
\end{equation*}
such that 
\[
\Psi_n(|\nabla u (x,y)|)\le h(x,y)  \quad \text{ for all } n\ge 2.
\]	
Applying Lebesgue's dominated convergence, we can take the limit as $n\to +\infty$ in the right-hand side of \eqref{ineq:gepsilon}. Arguing similarly for the left-hand side, we deduce \eqref{desyoung}.

\medbreak

When $\| \nabla u \|_{L^{\infty}(\mathbb{R} \times \omega;\mathbb{R}^N)} < 1$,
 we can choose  $n$ large enough so that  
\[
\Psi_n(|\nabla u(x,y)|) = 1 - \sqrt{1 - |\nabla u(x,y)|^2}  \quad \text{ for all } (x,y)\in\bbR\times\omega
\]
and the conclusion follows immediately from Theorem \ref{theoremAlberti}. In addition,  in this case, since the right-hand side of (\ref{desyoung}) is finite, we also deduce that equality holds in (\ref{desyoung})
 if and only if there exists $a\in\R$ such that $u (x+a,y)= u^{\star}(x)$ for every $(x,y)\in\R\times\omega$.
\end{proof}

We can now deduce the following proposition.

\begin{proposition}  \label{existenceofrearrangement}
For all $u \in \mathcal{X}$, taking values in $[-1,1]$, and satisfying $|\nabla u|\in L^{2}(\mathbb{R} \times \omega)$, there exists $u^{\star} \in \mathcal{X}$ (depending only on $x \in \mathbb{R}$) such that 
\[
\mathcal{J}(u^{\star}) \leq \mathcal{J}(u).
\]
Moreover, if $\| \nabla u \|_{L^{\infty}(\mathbb{R} \times \omega;\mathbb{R}^N)}<1$, we have equality if and only if there exists $a\in\R$ such that $u (x+a,y)= u^{\star}(x)$ for every $(x,y)\in\R\times\omega$.
\end{proposition}

\begin{proof}
Let $u\in\calX$ taking values in $[-1,1]$ and $u^\star\in\calX$ be its monotone rearrangement as defined above. 
Combining Theorem \ref{theoremyoung} and Cavalieri's principle, see for example \cite[Theorem 2.6]{Alberti2000}, we obtain the statement of the theorem.
\end{proof}


\subsection{Existence of a one-dimensional  minimizer}	
\label{subsection:1Dminimizer}


Let us first  notice that the functional $\mathcal{J}$ is bounded from below in $\calX$, as $\mathcal{J}(u)\ge0$ for all $u \in \calX$.

\begin{lemma} \label{existence1Dseq}
There exists a minimizing sequence $\left( u_n \right)_n\subset \calX$ satisfying $u_n(x,y)=u_n(x)$ for all $(x,y)\in \bbR\times\omega$ and all $n\in\bbN_0$. Moreover for every $n\in\N_0$, $u_n$ is a nondecreasing function taking values in $[-1,1]$ and such that $u_n(0)=0$.
\end{lemma}

\begin{proof}
Let $\left( v_n \right)_n\subset \calX$ be a minimizing sequence of $\mathcal{J}$, that is,
$\mathcal{J}(v_n) \to \inf_{\calX} \mathcal{J}$ as $n\to +\infty$.
Defining $u_n=\sup (-1,\inf(v_n,1))$ and observing that $\mathcal{J}(u_n)\le \mathcal{J}(v_n)$, we have $-1\le u_n \le 1$ for all $n\in\bbN_0$. The lemma is now an obvious consequence of Proposition \ref{existenceofrearrangement}.
\end{proof}

Consequently, for the minimizing sequence of Lemma \ref{existence1Dseq}, we have $\mathcal{J}(u_n)=m_{N-1}(\omega) \mathcal{J}_1(u_n)$, where
\begin{equation*}
\mathcal{J}_1(u) = \int_{- \infty}^{+ \infty} \left( 1 - \sqrt{1 - \abs{u^{\prime}(t)}^{2}} + W(u(t)) \right) \mathrm{d}t.
\end{equation*}

We will prove that this functional attains its minimum in the corresponding one-dimensional space $\mathcal{X}_1$, defined in \eqref{spaceX1D}.

\begin{proposition}	\label{prop:existence1Dmin}
$\mathcal{J}_1$ attains its minimum in $\calX_1$.
\end{proposition}

\begin{proof}
Consider a minimizing sequence $\left( u_n \right)_n\subset \calX_1$ as given by Lemma \ref{existence1Dseq}.
Observe that for each $\varepsilon>0$, there is $t_n>0$ such that $u_n(t_n) = 1-\varepsilon$, $0\le u_n(t) \le 1-\varepsilon$ for all $t\in  [0,t_n]$ and $u_n(t)>1-\varepsilon$ for all $t\in\ (t_n,+\infty)$. The sequence ${(t_n)}_n$ is bounded, since
\begin{equation*}
  t_n \cdot \min_{[0,1-\varepsilon]} W   \le \mathcal{J}(u_n) \le  C,
\end{equation*}
for some positive constant $C$ independent from $n$. 
A similar argument shows that there exists a bounded sequence ${(s_n)}_n$ such that
$s_n<0$, $u_n(s_n) = -1+\varepsilon$, $-1+\varepsilon\le u_n(t) \le 0$ for all $t\in  [s_n,0]$ and $u_n(t)<-1+\varepsilon$ for all $t\in\ (-\infty,s_n)$. 

From \eqref{quadraticEstimates}, it follows that $\sup_n\norm{u_n^\prime}_{L^2(\mathbb{R})}$ is finite. In addition, since $\norm{u_n^\prime}_{L^\infty(\mathbb{R})} \le 1$ and
\begin{equation*}
\abs{u_n(t_2)-u_n(t_1)} = \abs{\int_{t_1}^{t_2} u_n^\prime(s)\,ds} \le \abs{t_2-t_1}, 	\qquad\forall\ n\in\bbN_0, \ \forall\ t_1,t_2\in\bbR,
\end{equation*}
the sequence ${(u_n)}_n$ is equi-continuous.
Therefore, by Ascoli-Arzel\`a Theorem and the uniform boundedness of ${(u'_n)}_n$ in $L^2\cap L^\infty$, there exists $u\in W^{1,\infty}(\bbR)$ such that, up to a subsequence,
\begin{gather*}
u_n  \longrightarrow u  \text{ uniformly on every compact set } K\subset \bbR,\\
u_n^\prime \rightharpoonup u^\prime  \text{ in } L^2(\bbR),
\end{gather*}
as $n \to + \infty$. Moreover, from the uniform convergence we infer  that $u$ is nondecreasing, $u(t)\in [-1,1]$ for all $t\in\bbR$ and $u(0)=0$. Considering another subsequence if necessary, we have from the boundedness of ${(t_n)}_n$ and ${(s_n)}_n$ that $t_n\to \overline{t} >0$  and  $s_n\to \overline{s} <0$, and $u(t) > 1-\varepsilon$ for all $t > \overline{t}$ while $u(t) < - 1+\varepsilon$  for all $t < \overline{s}$.

We now prove that $u\in\calX_1$ and is a minimizer of  $\mathcal{J}_1$. Since $W$ is nonnegative, applying Fatou's Lemma, we deduce that
\begin{equation*}
 \int_{-\infty}^{+ \infty} W(u) \, \mathrm{d}t    \le  \liminf_{n\to+\infty}  \int_{-\infty}^{+ \infty} W(u_n) \, \mathrm{d}t   .  
\end{equation*}

On the other hand, the convexity of the function $g(t) = 1-\sqrt{1-t^2}$ 	implies the weak lower semi-continuity of the volume integral
\[
K_1(u) = \int_{-\infty}^{+ \infty}  \left(1-\sqrt{1-|u^\prime|^2}\right)  \mathrm{d}t,
\]
that is,
\begin{equation*}
\text{ if } \quad  u_n^\prime \rightharpoonup u^\prime \text{ in } L^2(\bbR) , \qquad \text{ then } \quad K_1(u) \le \liminf_{n \to + \infty}   K_1(u_n) .
\end{equation*}
Therefore one easily concludes that $\mathcal{J}_1(u)\le  \inf_{\calX_1} \mathcal{J}_1$.

It remains to prove that  $u\in\calX_1$. From the definitions of  $\overline{t}$  and  $\overline{s}$  and the monotonicity of $u$,  it is clear that 
\begin{equation*}
\lim_{t\to -\infty} u(t) = -1 \quad \text{and} \quad \lim_{t\to +\infty} u(t) = +1 .
\end{equation*}
Since for all $n\in\bbN_0$, the function $u_n$ is Lipschitz continuous with Lipschitz constant $L_n\le 1$, we conclude that $u$ is  also  a Lipschitz function with Lipschitz constant $L\le 1$. Hence, by Rademacher's Theorem, $u$ has a classical derivative a.e.~in $\bbR$ and $\norm{u^\prime}_{L^{\infty}(\R)} = L \le 1$.
\end{proof}

We proved the existence of a  minimizer of $\mathcal{J}_1$ in $\calX_1$ and therefore of a one-dimensional minimizer of $\mathcal{J}$ in $\calX$. Our  next step will be to show that this minimizer satisfies $\norm{ \nabla u }_{L^\infty(\R \times \omega; \R^N)}  =  \| u' \|_{L^\infty(\R)} < 1$, which allows to recover the Euler-Lagrange equation \eqref{eqNdim}.

To do so, we introduce a family of auxiliary functions obtained by stretching the graph of a given function $u\in \calX_1$ on a bounded interval. We anticipate that we will apply this construction to a minimizer of $\mathcal{J}_1$. 

\begin{definition}  \label{utheta}
Let $u\in\calX_1$.
For any  points $t_0 < t_1$ and any $0 < \theta < 1$, we define the function $u_\theta : \bbR \to \bbR$ by
\begin{align*} 
u_{\theta}(t) = \left\{ \begin{aligned}
			& u(t) \quad & t < t_0  ,   \\
			& u\big((1-\theta)(t - t_0) + t_0\big) \quad & t_0 \leq t < \overline{t_1}  ,   \\
			& u\Big(t - \frac{\theta}{1 - \theta}(t_1-t_0) \Big) \quad & t \geq \overline{t_1}  ,
	\end{aligned} \right.
\end{align*}
where $\overline{t_1} = t_0 + \frac{t_1- t_0}{1 - \theta} > t_1$. 
\end{definition}
From its definition, it is easily seen  that $u_\theta \in \calX_1$ because we eventually lower the value of the Lipschitz constant of $u$.

\begin{proposition} \label{proposition_1D}
If $u$ is a minimizer of $\mathcal{J}_1$ in $ \calX_1$ then $\norm{u^\prime}_{L^\infty(\R)} < 1$ and $u$ is the unique, up to translations, classical solution of the boundary value problem	\eqref{eqNdim}--\eqref{boundarycondNdim}. Moreover $u$ satisfies \eqref{cons.energy}.
\end{proposition}

\begin{proof}
Let $u$ be a minimizer of $\mathcal{J}_1$ in $\calX_1$. Then $\mathcal{J}(u)\le \mathcal{J}(v)$ for all $v\in \calX_1$. In particular, this inequality holds for $v=u_\theta$ obtained from $u$ by stretching it in any bounded interval as described above. 
We will  compare the energy of  $u$  and  $u_\theta$. First, notice that 
\begin{equation*}
\begin{aligned}
K_1(u_\theta) 	&= \int_{-\infty}^{+ \infty} \left(1-\sqrt{1-\abs{u_\theta^\prime}^2}\right) \mathrm{d}t   \\
			&=  \int_{-\infty}^{t_0} \left(1-\sqrt{1-\abs{u_\theta^\prime}^2}\right) \mathrm{d}t  +  \int_{t_0}^{\overline{t_{1}}} \left(1-\sqrt{1- \abs{u_\theta^\prime}^2}\right) \mathrm{d}t   +  \int_{\overline{t_{1}}}^{+\infty} \left(1-\sqrt{1- \abs{u_\theta^\prime}^2}\right) \mathrm{d}t .
\end{aligned}
\end{equation*}
Applying the  change of variables $s=(1-\theta)(t - t_0) + t_0$ to the second integral, we get
\begin{equation*}
\begin{aligned}	
\int_{t_0}^{\overline{t_{1}}} \left(1-\sqrt{1-\abs{u_\theta^\prime}^2}\right) \mathrm{d}t  &=    \int_{t_0}^{\overline{t_{1}}} \left(1-\sqrt{1-(1-\theta)^2  \abs{u^\prime \big((1-\theta)(t - t_{0}) + t_{0}\big) }^2 } \right) \mathrm{d}t        \\
 &= \frac{1}{1-\theta}\int_{t_0}^{t_1} \left(1-\sqrt{1-(1-\theta)^2 \abs{u^\prime(s)}^2 }\right)  \mathrm{d}s.
\end{aligned}
\end{equation*}
As for the third integral, the equality
\begin{equation*}
\int_{\overline{t_{1}}}^{+\infty} \left(1-\sqrt{1- \abs{u_\theta^\prime}^2}\right) \mathrm{d}t  =  \int_{t_1}^{+\infty} \left(1-\sqrt{1- \abs{u^\prime}^2}\right) \mathrm{d}t 
\end{equation*}
is a simple consequence of invariance of the integral by translation. Hence, we get
\[
\begin{aligned}
K_1(u_\theta)-K_1(u)
&= \frac{1}{1-\theta}\int_{t_0}^{t_1}  \left( 1-\sqrt{1-(1-\theta)^2 \abs{u^\prime}^2} \right)  \mathrm{d}t  - \int_{t_0}^{t_1}  \left(1-\sqrt{1- \abs{u^\prime}^2} \right) \mathrm{d}t  \\
&=\frac{\theta}{1-\theta} (t_1-t_0) + \frac{\theta(\theta-2)}{1-\theta} \int_{t_0}^{t_1} \frac{1}{(1-\theta)\sqrt{1- \abs{u^\prime}^2}+\sqrt{1-(1-\theta)^2 \abs{u^\prime}^2}}  \,\mathrm{d}t .
\end{aligned}
\]

Applying the same arguments to the potential energy, we conclude that
\begin{equation*}
\int_{-\infty}^{+ \infty}  W(u_\theta)\,\mathrm{d}t - \int_{-\infty}^{+ \infty}  W(u)\,\mathrm{d}t  =  \frac{\theta}{1-\theta} \int_{t_0}^{t_1}  W(u)\,\mathrm{d}t .  
\end{equation*}

Then
for any  $\theta\in\ (0,1)$ and  any  $t_0<t_1$, we  have 
\begin{gather*}
\frac{\theta}{1-\theta}  \bigg(1 + \frac{\theta-2}{t_1-t_0} \int_{t_0}^{t_1} \frac{1}{(1-\theta)\sqrt{1- \abs{u^\prime}^2}+\sqrt{1-(1-\theta)^2 \abs{u^\prime}^2}}  \,\mathrm{d}t + \frac{1}{t_1-t_0} \int_{t_0}^{t_1}  W(u)\,\mathrm{d}t  \bigg)\\
= \frac{\mathcal{J}(u_\theta)-\mathcal{J}(u)}{t_1-t_0}\ge 0.
\end{gather*}

Since $u$ is a Lipschitz function, it has a classical derivative a.e.~in $\bbR$. Then applying Lebesgue's differentiation Theorem,
we may take the limit as $t_1\to t_0$ on the left-hand side of the equality above
and we obtain
\begin{multline*} 
\lim_{t_1\to t_0} \frac{1}{t_1- t_0}   \int_{t_0}^{t_1} \frac{1}{(1-\theta)\sqrt{1- \abs{u^\prime (t)}^2 }+\sqrt{1-(1-\theta)^2 \abs{u^\prime (t) }^2 }}  \,\mathrm{d}t   \\
	=   \frac{1}{(1-\theta)\sqrt{1- \abs{u^\prime(t_0)}^2 }+\sqrt{1-(1-\theta)^2 \abs{u^\prime(t_0)}^2 }}    \quad \text{a.e. } t_0 \in\bbR 
\end{multline*}
and
\begin{equation*}
\lim_{t_1\to t_0} \frac{1}{t_1- t_0} \int_{t_0}^{t_1}  W(u(t))\,\mathrm{d}t  =  W(u(t_0))   \quad \text{a.e. } t_0 \in\bbR .
\end{equation*}
Therefore, for almost every~$t_0 \in\bbR$ and any fixed  $\theta\in\ (0,1)$, we have
\begin{equation*}
\begin{gathered}
 \frac{\theta}{1-\theta} + \frac{\theta(\theta-2)}{1-\theta} \frac{1}{(1-\theta)\sqrt{1- \abs{u^\prime(t_0)}^2 }+\sqrt{1-(1-\theta)^2 \abs{u^\prime(t_0)}^2 }}  +  \frac{\theta}{1-\theta} W(u(t_0)) =
\\
 \lim_{t_1\to t_0} \frac{\mathcal{J}(u_\theta)-\mathcal{J}(u)}{t_1- t_0}  \ge 0,
\end{gathered}
\end{equation*}
which yields the inequality
\begin{equation*} 
	(1-\theta)\sqrt{1 - \abs{u^\prime}^2} + \sqrt{1-(1-\theta)^2 \abs{u^\prime}^2} \ge  \frac{2-\theta}{W(u)+1}   \quad \text{a.e.~in } \bbR.
\end{equation*}
Taking the limit as $\theta \to 0^{+}$, we obtain the following quantitative bound on the derivative of $u$ 
\begin{equation*} 
\abs{u^\prime} \le \sqrt{1- \frac{1}{\big(1+W(u)\big)^2}} <1   \quad \text{a.e.~in } \bbR,	
\end{equation*}
from which we can obtain the uniform estimate 
\begin{equation*} 
\abs{u^\prime} \le \sqrt{1- \frac{1}{\big(1+\max_{[-1,1]} W\big)^2}} <1   \quad \text{a.e.~in } \bbR,
\end{equation*}
and hence conclude that
\[
\norm{u'}_{L^\infty(\R)} <1.
\]
The fact that $u$ is a classical heteroclinic solution of the autonomous scalar equation \eqref{eqNdim}  now follows from standard arguments. Notice that once we know that the minimizer $u$ satisfies the strict inequality $\|u^\prime\|_{L^\infty(\R)} < 1$, the previous argument is reminiscent of a standard variation in the spirit of the Dubois-Reymond Lemma and yields directly (\ref{cons.energy}). Observe also that as soon as $\|u^\prime\|_{L^\infty(\R)} < 1$, it is a standard issue to prove the smoothness of the solution. 
\end{proof}

\begin{remark} \label{rem:grad1not1D}
As a consequence of Proposition \ref{proposition_1D}, we see that	if $u \in \calX$ is a minimizer of $\mathcal{J}$ and $\norm{\nabla u}_{L^\infty(\R \times \omega;\R^N)} = 1$  then $u$ cannot be one-dimensional.
\end{remark}

\begin{remark}
We can also deduce the previous proposition from a gradient estimate from \cite[Theorem 4.1, Corollary 4.3]{BartnikSimon1983}. However this argument does not give a quantitative bound and relies moreover on more delicate estimates. We believe the argument we provide is the more direct one. 
\end{remark}

So far we proved the existence of a one-dimensional minimizer of $\mathcal{J}$ with $\norm{\nabla u}_{L^\infty(\R \times \omega;\R^N)} < 1$. From Proposition \ref{existenceofrearrangement}, we see that any minimizer  of $\mathcal{J}$ with $\norm{\nabla u}_{L^\infty(\R \times \omega;\R^N)} < 1$ depends only on the first variable. But our argument does not exclude the existence of a minimizing sequence of $\mathcal{J}$ converging to some $u_0\in\calX$ with $\norm{ \nabla u_0 }_{L^\infty(\R \times \omega;\R^N)} = 1$. We will use  a regularization argument inspired from \cite{CoelhoCorsatoObersnelOmari2012,BonheureDerletDeCoster2012} to rule out this possibility.

\begin{proof}[Proof of Theorem \ref{theoremfinal}]
Arguing as in Propositions \ref{prop:existence1Dmin}  and  \ref{proposition_1D}, we see that there exists a minimizer $u\in\calX$ of $\mathcal{J}$ with $\norm{ \nabla u }_{L^\infty(\R \times \omega;\R^N)} < 1$ which is a classical smooth solution of the boundary value problem  \eqref{eqNdim}--\eqref{boundarycondNdim} and satisfies the conservation of energy law \eqref{cons.energy}. The proof will be complete once we can exclude the existence of minimizers $u$ of $\mathcal{J}$ with $\| \nabla u \|_{L^\infty(\R \times \omega;\R^N)} = 1$.
 
We introduce the  modified functional  $\mathcal{J}_n \colon  W^{1,\infty}(\bbR\times\omega)  \to [0, +\infty]$ given by
\begin{equation*}
	\mathcal{J}_n (u) = \int_{\mathbb{R} \times \omega}  \left( \Psi_n(| \nabla u |) + W(u)  \right)  \, \mathrm{d}\bar{x},
\end{equation*}
where the function $\Psi_n$ is given  in Definition \ref{gtronqué}. The functional $\mathcal{J}_n$ is not singular since the function $\Psi_n$ is defined in $\bbR$. Also  notice that  for all $u\in\calX$ and all $n \in \mathbb{N}_0$,
\[
\mathcal{J}_n (u) \leq \mathcal{J}(u),
\]
and  by the definition of  $\Psi_n$ the inequality is strict if $\norm{ \nabla u }_{L^\infty(\R \times \omega;\R^N)} = 1$,
since for every $\eps>0$ there exists a set $E \subset \bbR\times\omega$ with positive measure such that $\abs{\nabla u(x)} > 1- \eps$  for a.e.~$x \in E$.

Moreover, Theorem \ref{theoremAlberti} can be applied to this  functional.
So, arguing as above, 
we conclude that $\mathcal{J}_n$ has a unique (up to translations) minimizer $u_n \in \calX$
that depends only on the first variable, which is nondecreasing as a function of this variable and satisfies $u_n(0) = 0$. 
Besides,  $u_n$ is a classical  $C^2$ solution of the equation
\[
\big( \psi_n(u^{\prime}) \big)^{\prime}  =  W'(u),
\]
where $\psi_n  =  \Psi'_n$ and $u_n$ satisfies 
\[
\lim_{x \to \pm \infty} u(x) = \pm1.
\]
Moreover, any solution $u$ of this equation satisfies  the conservation of energy law
\[
 \Upsilon_n \big( \psi_n(u') \big) - W(u)	= E_n ,
\]
for some $E_n\in\bbR$  and
with $\displaystyle \Upsilon_n(y) = \int_0^y \psi_n^{-1} (\xi) \, d\xi $.
From the definition of  the modified operator, we have
\begin{align*}
\Upsilon_n \big( \psi_n(t) \big) 	=	\left\{ \begin{aligned}
	& \frac{1}{\sqrt{1 - t^2}} -  1	 \quad	& \text{ if }  t^2 \le 1-\frac{1}{n^2}\\
	& \tilde a_n  + \tilde b_n \bigg( t^2 - 1+\frac{1}{n^2} \bigg)   +  \tilde c_n \bigg( t^4 - \Big( 1-\frac{1}{n^2} \Big)^2 \bigg)	\quad	& \text{ if }  t^2 > 1-\frac{1}{n^2}
	\end{aligned} \right.
\end{align*}
with  $\tilde a_n = n-1$, $\tilde b_n = b_n - 2 c_n \big( 1-\frac{1}{n^2} \big)$, $\tilde c_n = 3c_n$  and  $a_n$, $b_n$ and $c_n$ being as in Definition \ref{gtronqué}.

Then,   as  $u_n$  satisfies the boundary conditions $\displaystyle \lim_{x \to \pm \infty} u(x) = \pm1$,  the conservation of energy law for $u_n$ becomes
\begin{equation}
			\label{aux:consE-modif}
\frac{1}{\sqrt{1 - \abs{u'_n}^2}} -  1 -  W(u_n)  = E_n ,
\end{equation}
as long as   $\abs{u'_n}^2$ remains sufficiently small. In particular, there is a sequence ${(t_k)}_k$  such that $t_k \to \pm\infty$  and  $u'_n(t_k) \to 0$,   as $k \to +\infty$. Then, taking the limit along this sequence as  $k \to + \infty$ in~\eqref{aux:consE-modif}, we see that $E_n = 0$ for all $n \in \mathbb{N}_0$. So, \eqref{aux:consE-modif} gives
\[
 \Upsilon_n \big( \psi_n(u'_n) \big) - W(u_n)	= 0
\]
and for $\abs{t}$  sufficiently large, we obtain
\[
\frac{1}{\sqrt{1 - \abs{u'_n}^2}} -  1  - W(u_n)	= 0.
\]
Take $n$ sufficiently large so that 
\[
\ell := 1 - \frac{1}{1+\max_{[-1,1]} W}   <  1 - \frac{1}{n^2}.
\]
As long as $|u_n'|^2\le  1 - \frac{1}{n^2}$, we have actually $|u_n'|^2\le\ell$ . But on the other hand, if $|u_n'(t_0)|^2\le\ell$, then $|u_n'(t)|^2 < 1 - \frac{1}{n^2}$ on a neighborhood of $t_0$, which implies
\[
\Upsilon_n \big( \psi_n(u'_n) \big) = \frac{1}{\sqrt{1 - \abs{u'_n}^2}}-1
\]
and therefore the bound $|u_n'(t)|^2\le\ell$ still holds in this neighborhood.
It follows that the equality
\[
\abs{u'_n}^2  =  1 - \frac{1}{1+W(u_n)}  \le   \ell	
\]
holds everywhere. 
In this way, we obtain an uniform bound for~$\abs{u'_n}$	and 
conclude that	$u_n \in \calX$.

Our choice of $n$ then implies that 
\[
\mathcal{J}(u_n)  =  \mathcal{J}_{n}(u_n).
\]
Now, let us  assume that $u_{0} \in \calX$ is a minimizer of $\mathcal{J}$ with $\norm{ \nabla u_0 }_{L^\infty(\R \times \omega;\R^N)}=1$. From Remark~\ref{rem:grad1not1D}, $u_0$ is not one-dimensional. 	 
This means, in particular, that  $u_0$ is not a minimizer of $\mathcal{J}_n$ by our previous discussion, but also as a consequence of  Theorem~\ref{theoremAlberti}.	
Then  we obtain 
\begin{equation*}
\mathcal{J}(u_{n}) = \mathcal{J}_{n} (u_{n}) < \mathcal{J}_{n}(u_{0}) < \mathcal{J}(u_{0}),
\end{equation*}
which contradicts the fact that $u_0$ minimizes $\mathcal{J}$.
\end{proof}

We conclude  this section  with an example for which we can compute an explicit solution.
\begin{example}
By analogy with the classical Allen-Cahn equation (see \cite{Carbou1995,Alberti2000}), we are able to compute the exact solutions to \eqref{eqNdim}--\eqref{boundarycondNdim} when $W$ is given by
\[
W(u) = -1 + \sqrt{\dfrac{2}{2 - {(1-u^2)}^2}}.
\]
Easy computations show that for this potential the  conservation of energy law  simplifies to
\[
- \frac{\abs{u'}^2}{2} + \frac{1}{4} {(u^{2}-1)}^{2}  = 0, 
\]
the same as for the Allen-Cahn equation
and therefore
 all solutions are of the form 
\[
u(t)= \tanh \left(\dfrac{t+a}{\sqrt{2}}\right) \!, 
\] 
where $a$ is a real parameter.
\end{example}

\subsection{One-dimensionality for non-autonomous potentials of the type \texorpdfstring{$W(x,u)$}{W(x,u)}}

In this section, we consider a non-autonomous problem with a potential that depends spacially on the first variable only, that is, a potential of the type  $W(x,u)$. In this situation we can still prove  one-dimensionality of the minimal transition with a quite short argument. Let us consider $\mathcal{G} : \calX \to [0,+\infty]$ and $\mathcal{G}_1 : \calX_1 \to [0,+\infty]$ defined respectively by
\[
\calG(u) = 	\int_{\bbR\times\omega} \left( 1 - \sqrt{1 - |\nabla u|^{2}} + W(x,u) \right) \mathrm{d}\bar{x}
\]
and
\[
\mathcal{G}_1(u) = \int_{-\infty}^{+ \infty} \left( 1 - \sqrt{1- |u'|^2} + W(x,u) \right) \, \mathrm{d}x.
\]

\begin{theorem} \label{thm:nouveau1D}
Assume that $W(\cdot,s) \in L^1_{loc}(\bbR)$ for all $s\in\bbR$ and  $W(x,\cdot)$  satisfies $(W_2)$ for all $x\in\bbR$.
Then
\[
\inf_{\calX} \calG = m_{N-1}(\omega) \inf_{\calX_1} \mathcal{G}_1.
\]
Moreover,  if $ \inf_{\calX_1} \mathcal{G}_1 $ is attained,  then $ \inf_{\calX} \calG $ can only be achieved by  a one-dimensional minimizer of $\mathcal{G}_1$.
\end{theorem}

\begin{proof}
Let ${(u_n)}_n \subset \calX_1$ be a minimizing sequence of $\mathcal{G}_1$.  Taking $u_n(x,y) = u_n(x)$ for all $n\in \bbN_0$, we see that
\begin{equation*}
\calG(u_n) = m_{N-1}(\omega) \int_\bbR  \left( 1 - \sqrt{1 - |u'_n|^{2}} + W(x,u_n(x)) \right) \mathrm{d}x	\longrightarrow		m_{N-1}(\omega)  \inf_{\calX_1} \mathcal{G}_1.	
\end{equation*}
So,
\[
m_{N-1}(\omega)  \inf_{\calX_1} \mathcal{G}_1	\ge	\inf_{\calX} \calG.
\]
On the other hand,	by Fubini's Theorem and the monotonicity of $t \to 1 - \sqrt{1 - t^2}$  we have
\begin{equation}
\begin{aligned}
		\label{aux.calG}
\calG(u) &=	\int_{\bbR\times\omega} \left( 1 - \sqrt{1 - |\nabla u|^{2}} + W(x,u) \right) \mathrm{d}\bar{x}		\\
		 &= 	\int_\omega \int_\bbR   \left( 1 - \sqrt{1 - |\nabla u|^{2}} + W(x,u) \right)   \, \mathrm{d}x  \, \mathrm{d}y			\\
		 &\ge   \int_\omega \int_\bbR   \left( 1 - \sqrt{1 - |\partial_x u|^{2}} + W (x,u) \right)   \, \mathrm{d}x  \, \mathrm{d}y		\\
		 &=   \int_\omega \mathcal{G}_1 \big(u (\cdot,y)\big)  \, \mathrm{d}y	\\
		 &\ge   \int_\omega \inf_{\calX_1} \mathcal{G}_1   \, \mathrm{d}y		\\		
		 &=	m_{N-1}(\omega) \inf_{\calX_1} \mathcal{G}_1,
\end{aligned}
\end{equation}
for all $u\in \calX$. 
This proves the first statement of the theorem. Now, let us assume (by contradiction) that $\calG$ has a minimizer $u$ in $\calX$ which is not one-dimensional.
Then there is a set $B\subset \bbR\times\omega$, with positive Lebesgue measure, such that $\nabla_y u \not= 0$ a.e.~in $B$. 
Then the first inequality in \eqref{aux.calG} is strict for $u$, since
\[	
\abs{\nabla u}^2 = \abs{\partial_x u}^2 + \abs{\nabla_y u}^2	>  \abs{\partial_x u}^2	,  \quad \text{ in } B.
\]
This leads to a contradiction.

\end{proof}


\section{Heteroclinic solutions for non-autonomous equations}	
\label{section:nonAutonomous} 


This section concerns the study of the non-autonomous functional $\mathcal{L}_a$ given in \eqref{functionalL}. As before, we will look for minimizers of $\calL_a$ in the functional space $\calX_1$, given in \eqref{spaceX1D}, and discuss the existence of  solutions to the associated Euler-Lagrange equation.
In particular, we will discuss assumptions on the weight $a(t)$ that ensure the existence of minimizers of $\calL_a$. In particular, an additional symmetry assumption will also be considered.
Theorem \ref{thm:nouveau1D} motivates to work with a one-dimensional situation only.

\subsection{Regularity of the minimizer and solution of the Euler-Lagrange equation}

We first work out the regularity of minimizers. A difference with respect to the autonomous situation is that we lose the conservation of energy. We will therefore argue differently to derive the a priori bound on the derivative.

The following lemma tells us that assumptions $(a_1^\prime)$ and $(W_1)$ are sufficient to deduce that minimizers are weak solutions of the Euler-Lagrange equation, as well as to derive a local bound on the derivative. 

\begin{lemma} \label{lemma:solutionC1}
Assume $(W_1)$, $(a_1^\prime)$. If $u_a \in \mathcal{X}_1$ is a minimizer of $\mathcal{L}_a$, then $u_a$ is a $C^1 \cap W^{2,2}_{\text{loc}}$ solution of 
\[
\left( \frac{u'}{\sqrt{1 - |u'|^2}} \right)' = a(t) W'(u),  \quad t \in \mathbb{R}
\]
and $\| u_a' \|_{L^\infty_{\text{loc}}(\R)} < 1$
\end{lemma}

\begin{proof}
To prove the regularity, we will argue on compact sets of $\mathbb{R}$ and proceed in a similar way as in \cite{MawhinBrezis2010,BereanuJebeleanMawhin2011-2,BonheureDaveniaPomponio2015}. Let $M > 1$ be any large number. If we denote by
\[
\mathcal{L}_{a}^M(u) = \int_{-M}^M \left( 1 - \sqrt{1 - |u'|^2} + a(t) W(u) \right) \, \mathrm{d}t
\] 
and
\[
\mathcal{X}_M := \left\{ u \in \mathcal{X}_1 \, : \, u(t) = u_a(t) \, \text{ for } t \in \mathbb{R} \setminus (-M, M) \right\},
\]
we immediately get that
\begin{equation}	 \label{eq:localmin}
\mathcal{L}_{a}^M (u_a) \leq \mathcal{L}_{a}^M(u), \quad \text{ for all } u \in \mathcal{X}_M,
\end{equation}
by the minimality of $u_a$. Next, we define the linearized functional
\[
\mathcal{E}_a^M(u) = \int_{-M}^M \left( 1 - \sqrt{1 - |u'|^2} + a(t) W'(u_a) u \right) \, \mathrm{d}t.
\]
By hypothesis $(W_1)$ and $(a_1^\prime)$ and the fact that $u_a \in \mathcal{X}_1$, this functional is bounded from below in $\mathcal{X}_M$. Moreover, it is also strictly convex. Therefore it has an unique minimizer $u_M \in \mathcal{X}_M$.

We claim that $u_M = u_a$.
Using \eqref{eq:localmin} and the convexity of the kinetic part 
\[
K_M(u) = \int_{-M}^M \left( 1 - \sqrt{1 - |u'|^2} \right) \, \mathrm{d}t,
\]
we get for $\varepsilon > 0$
\[
\mathcal{L}_a^M(u_a) \leq \mathcal{L}_a^M \big( (1 - \varepsilon) u_a + \varepsilon u_M \big) \leq (1 - \varepsilon) K_M (u_a) + \varepsilon K_M(u_M) + \int_{-M}^M a(t) W \big( (1-\varepsilon) u_a + \varepsilon u_M \big) \, \mathrm{d}t.
\]
Therefore,
\[
K_M(u_M) \geq K_M(u_a) + \int_{-M}^M a(t) \frac{1}{\varepsilon} \Big( W(u_a) - W \big( u_a + \varepsilon (u_M- u_a) \big) \Big) \, \mathrm{d}t.
\]
Thanks to hypothesis $(W_1)$, $(a_1^\prime)$ and the fact that $u_a , u_M  \in \mathcal{X}_1$, we can take the limit as $\varepsilon\to 0$ to obtain
\[
\mathcal{E}_a^M(u_a) \leq \mathcal{E}_a^M(u_M).
\]
Since the minimizer of $\mathcal{E}_a^M$ in $\mathcal{X}_M$ is unique, the claim is proved.

\medbreak

Finally, since $\lim_{t \to \pm \infty} u_a(t) = \pm 1$, we have that for $M$ large enough,
\[
\frac{u_a(M) - u_a(-M)}{2M} < 1.
\]
Therefore, thanks to \cite[Theorem 4.1, Corollary 4.3]{BartnikSimon1983}, see also \cite{CorsatoObersnelOmariRivetti2013,BereanuJebeleanMawhin2014}, we know that the minimizer $u_a$ of $\mathcal{E}_a^M$ in $\mathcal{X}_M$ is a $C^1((-M, M)) \cap W^{2,2}((-M,M))$ solution of the Euler-Lagrange equation associated to the functional $\mathcal{E}_a^M$ and
\[
\| u_a' \|_{L^\infty((-M,M))} < 1.
\]
\end{proof}

\begin{remark}
The argument in the proof of Lemma \ref{lemma:solutionC1} is local. To obtain an uniform quantitative bound, we can proceed as in Proposition \ref{proposition_1D}. However, we need some more restrictive assumptions on $a(t)$, namely  $(a_1)$ and $a\in C^{0,1}(\mathbb{R})$ with
\[
a'(t) \leq C a(t), \quad \text{ for a.e. } t \in \R.
\]
In this case, we obtain the explicit quantitative bound
\[
\| u_a^{\prime} \|_{L^\infty(\R)} 	\leq \sqrt{1 - \frac{1}{  \left(1 + \|a\|_{L^\infty(\R)} \max_{[-1,1]}W + C \inf_{\mathcal{X}_1}\mathcal{L}_a \right)^2}}.
\]
\end{remark}

\subsection{Proof of Theorems \ref{thm.structural} and \ref{thm.periodic}}

Lemma \ref{lemma:limits} is a preliminary result giving a lower bound on the energy functional $\calL_a$ and will be useful for establishing the limit of minimizers at infinity.	
We state it without proof since it may be easily reconstructed from \cite[{\small \sc Fact 1} in \S 2.1]{BonheureSanchez2006} using   estimate~\eqref{quadraticEstimates}.

\begin{lemma}	\label{lemma:limits}
Assume $(a_1)$, $(W_1)$, $(W_2)$. Fix $\eps>0$ and set 
\[
\beta_\eps= \min \big\{  W(s) \mid 1 - \eps \le s \le 1 -\frac{\eps}{2} \text{ or }  -1+\frac{\eps}{2} \le s \le -1+\eps  \big\} .
\]
If $u\in\calX_1$ is such that there exist $t_1, t_2 \in \bbR$ for which $u(t_1)= 1 - \eps$ and $u(t_2)= 1 - \frac{\eps}{2}$ 
(or $u(t_1)= -1+\frac{\eps}{2}$ and $u(t_2)= -1+\eps$),
then we have
\[
\calL_a(u) \ge \abs{\int_{t_1}^{t_2} \left( 1 - \sqrt{1 - |u^{\prime}|^ 2} + a(t) W(u) \right) \, \mathrm{d}t } \ge \frac{\eps\sqrt{a_1 \beta_\eps}}{\sqrt{2}},
\]
where $a_1=\inf_{[t_1,t_2]} a(t)$.
\end{lemma}

\bigbreak

We are now ready to prove the main theorems of this section. Since we follow the lines of \cite[Theorem 2.4]{BonheureObersnelOmari2013}, we will not provide all the details.

\begin{proof}[Proof of Theorem \ref{thm.structural}]

Let $u_b$ be a minimizer of $\calL_b$. 
We may assume that $ \displaystyle \inf_{\calX_1} {\calL_a} < \min_{\calX_1} \calL_b$, otherwise, $u_a=u_b$ is a minimizer of $\calL_a$.
We divide the proof in three steps.

\medbreak

\noindent\textbf{Step 1. A  minimizing sequence.} 		
Notice  that   for all  $u\in\calX_1$,	$\calL_a(u)   \le  \calL_b(u) $. 

Arguing as in \cite[Theorem 2.4]{BonheureObersnelOmari2013}, we prove that for all $\overline{\eps} \in (0,1)$,
 there is some   $\eps\in \left(0,\frac{\overline{\eps}}{2}\right)$  such that 
there is a minimizing sequence ${(u_n)}_n \subset  \calX_1$, satisfying the following conditions for all $n \in \mathbb{N}_0$
\begin{itemize}
\item[$(i)$]  $-1\le u_n(t) \le 1$ for all $t\in\bbR$,
\item[$(ii)$]  there exist $s_n, t_n \in\bbR$, with $s_n<t_n$, such that 

\vspace{0.2cm}
	\begin{itemize}
	\item[$(a)$]   $u_n(s_n) = -1+\eps$, \ \ $u_n(t_n) = 1-\eps$ and 
	$-1+\eps \le u_n(t) \le 1-\eps$ for all $t\in[s_n,t_n]$,
	
	\vspace{0.2cm}
	\item[$(b)$]   $u_n(t) \le -1+\overline{\eps}$ for all $t\le s_n$ \  and \   $u_n(t) \ge 1-\overline{\eps}$ for all $t\ge t_n$,
	
	\vspace{0.2cm}
	\item[$(c)$]   $\displaystyle \int_{-\infty}^{s_n}  \left(  1 - \sqrt{1 -  \abs{{u_n}^{\prime}}^2} + a(t) W(u_n) \right) \mathrm{d}t  \le  \delta (\overline{\eps}) $   
	\ and \\
	$\displaystyle \int_{t_n}^{+\infty}  \left(  1 - \sqrt{1 - \abs{{u_n}^{\prime}}^2} + a(t) W(u_n)  \right)  \mathrm{d}t  \le  \delta (\overline{\eps}) $,   
	\end{itemize}
	
	\vspace{0.2cm}
\item[$(iii)$]  the sequences ${(s_n)}_n$ and ${(t_n)}_n$ are bounded,
\end{itemize}
where    $\delta (\overline{\eps})$ is a positive  $o(1)$   when $\overline{\eps} \to 0^+$.

\medskip
\noindent\textbf{Step 2. Convergence to a minimizer.} Arguing as in the proof of Proposition \ref{prop:existence1Dmin}, we  see that 
there is some $u\in W^{1,\infty}(\mathbb{R})$ such that,
up to a subsequence, 
\[
u_n  \to u  \text{ in } C^1_{loc}(\bbR)   \quad
\text{ and }  \quad   u_n^\prime \rightharpoonup u^\prime  \text{ in } L^2(\bbR).
\]
Combining the weak lower semicontinuity of the kinetic part $K_1$ with Fatou's Lemma, we infer 
\begin{equation*}
\calL_a(u) 	
= \int_\bbR \left(1-\sqrt{1-|u^\prime|^2} +  a(t) W(u) \right)  \mathrm{d}t  
		\le  \liminf_{n \to + \infty} \calL_a(u_n)   
= \inf_{\calX_1} \calL_a.
\end{equation*}
Furthermore,  considering another subsequence if necessary, we may assume that ${(s_n)}_n$ converges to some $\overline{s}\in\bbR$ and that ${(t_n)}_n$ converges to some $\overline{t}\in \bbR$ with $\overline{s}\le \overline{t}$.
In particular, 
we have
\[
-1 \le u(t) \le -1+\overline{\eps}  \ \text{ for all } t\le \overline{s}
\quad \text{ and } \quad
1-\overline{\eps}  \le u(t) \le 1 \  \text{ for all } t\ge \overline{t}.
\]
From  the estimate $ \int_\bbR  W(u) \, \mathrm{d}t  	\le M$, we deduce that $\displaystyle \liminf_{t \to - \infty} u(t) =  -1$ and $\displaystyle \limsup_{t \to + \infty} u(t) = 1$. Then  by Lemma~\ref{lemma:limits}, 
we see that 
\[
\lim_{t\to\pm\infty} u(t) =\pm1,
\]
so that   $u\in\calX_1$ and $\displaystyle	\calL_a(u) = \min_{\calX_1} \calL_a$.

\medskip
\noindent	\textbf{Step 3. The Euler-Lagrange equation.} 
Finally, thanks to Lemma \ref{lemma:solutionC1} we see that the minimizer of $\calL_a$ is a solution of \eqref{eq:nonAutonomous}. 
Since $\lim_{t \to \pm \infty} u(t) = \pm 1$ and $a(t)$ is bounded, we deduce from this equation that $\lim_{t \to \pm \infty} u'(t)$ exists. 
Moreover, since $u \in \mathcal{X}_1$, this limit can only be zero. 
This and the bound on $\|u'\|_{L^\infty((-M,M))} < 1$, for $M$ large enough, given in Lemma \ref{lemma:solutionC1} implies that $\| u' \|_{L^\infty(\mathbb{R})} < 1$.
\end{proof}

\bigbreak

We will be more sketchy for  the proof of Theorem \ref{thm.periodic}.

\begin{proof}[Proof of Theorem \ref{thm.periodic}] 
Arguing as in  the proof of Theorem \ref{thm.structural}, we  prove that
a minimizing sequence  ${(u_n)}_n$ of $\calL_a$ satisfies properties
$(i)$ and $(ii)$
and that  the sequence ${(t_n-s_n)}_n$
is bounded.
For each $n$, let $k_n\in\bbZ$ be such that $\hat{s}_n=s_n+k_nT \in [0,T)$. If $s_n \not\in[0,T)$, replace $u_n$ with $v_n(t)=u_n(t - k_nT)$.
Since $a$ is $T$-periodic, it is easily seen that $\calL_a(v_n)=\calL_a(u_n)$.
Moreover, $v_n$ satisfies conditions $(i)$ and $(ii)$,  with $s_n$ replaced by $\hat{s}_n$ and $t_n$ replaced by $\hat{t}_n = t_n+k_nT$.
Now, the sequence ${(\hat{s}_n)}_n\subset[0,T)$ is clearly bounded and therefore so is ${(\hat{t}_n)}_n$. The remaining of the proof is then similar.
\end{proof}


\subsection{Odd heteroclinics in a symmetric setting}	\label{subsection:oddHeteroclinics}


In this final section, we prove Theorem \ref{thm:AntiSym}. 
Now,  the functional $\calL_a$ satisfies additional  symmetry properties and we are looking for  antisymmetric solutions of problem \eqref{eq:nonAutonomous}--\eqref{eq:nonAutonomous2}.
Observe that two of the previous assumptions are now relaxed : $a(t)$ may change sign in a compact interval and may be unbounded at infinity. Here the new feature that will prevent from losses of compactness is the monotonicity of the weight $a(t)$. 

\begin{proof}[Proof of Theorem \ref{thm:AntiSym}]
We divide again the proof in several steps.

\medskip
\noindent	\textbf{Step 1. A minimizing sequence.}									
We observe that although $a(t)$ may change sign,  the functional $\calL_a^A$ is bounded from below  in $\calX_A$.
Indeed, if $a(t)$ takes negative values,	assumptions $(a_1^\prime)$, $(a_4)$, $(W_2')$ and
\[
0 \le W(u) \le \max_\bbR W = \max_{[-1,1]} W < + \infty
\] 
imply that there exists a constant $c \geq 0$ such that  for all $u\in\calX_A$
\[
\calL_a^A(u) =  \int_0^{T}   \Big[  1 - \sqrt{1 - |u^{\prime}|^{2}}  + a(t) W(u) \Big]  \, \mathrm{d}t	+	 \int_{T}^{+\infty}  \Big[ 1 - \sqrt{1 - |u^{\prime}|^{2}}  + a(t) W(u) \Big] \, \mathrm{d}t	\ge  - c.
\]
Fix $\varepsilon > 0$. We will prove that there is a minimizing sequence ${(u_{n})}_{n} \subset \mathcal{X}_{A}$ satisfying the following  properties for all $n\in\bbN_0$:
\begin{itemize}
\item[$(i)$]  $ u_n(t) \ge 0$ for all $t\ge0$,
\item[$(ii)$]  $ u_n(t) \le 1$ for all $t \ge 0$,
and either
\[
u_n(t) <1	\quad   \text{ for all }   t \ge T,
\]
or there is an unique  $s_n \geq T$	 such that 
\[
 u_n(t) < 1  \    \text{ for all }\  T \le t < s_n	 \quad \text{ and } \quad 	 u_n (t) = 1    \ 	\text{ for all } \  t \geq s_n ,
\]

\item[$(iii)$] there exists $t_n \ge T+1$  such that 
\[
0\le u_n(t) < 1-\varepsilon  \quad  \text{ for all } T+1 < t < t_n
\]	
and
\[
u_n(t) > 1 - \eps \quad \text{ for all }  t > t_n,
\]

\item[$(iv)$]   the sequence ${(t_n)}_n$ is bounded.
\end{itemize}

\vspace{0.2cm}
The first assertion $(i)$ is clear since $\calL_a^A(\abs{u_n})    =	\calL_a^A(u_n)$ by $(W_3)$.	

Next, let us define the function $v_n(t) =  \inf \big( u_n(t), 1 \big)$. Arguing as in the proof of Lemma \ref{existence1Dseq}, we see that 
\[
\calL_a^A(v_n)  \le  \calL_a^A(u_n)
\]
by  assumption ($W_2^\prime$) and therefore we may replace ${(u_n)}_n$ by ${(v_n)}_n$. Now,  suppose that  there is some  $t \ge T$  such that  $u(t) =1$. Then we define  $s_{n} \ge T$  as
\[
s_{n} = \inf  \big\{ t  \ge  T  \mid  u_n(t) = 1 \big\} 
\]
and consider the function
\begin{equation*}
w_n (t)   =	
 	\begin{cases}	
				u_{n}(t) 		&\text{ if } \,  0 \leq t < s_n, 
		\\[1mm]
				   1 			&\text{ if }  \, t \ge  s_n.
	\end{cases}
\end{equation*}
As $a(t) > 0$ for $t > T$, we see that	
\[
\calL_a^A(w_n)  \le  \calL_a^A(u_n),
\]
and $(ii)$ is established. 

Now,  suppose that $u_n(T+1) \ge 1 - \varepsilon$. In this case,  we set   
\[
t_n = T + 1.
\]
If  $u_n(t) > 1 - \varepsilon$ for all $t > T + 1$,  the function $u_n$ satisfies $(iii)$. Otherwise, set
\[
t^1_n   =   \sup  \big\{  t \geq T + 1 \mid  u_n(t) = u_n(T+1) \big\} 
\]
and define  the function $v^1_n \in \calX_A$ as
\begin{equation*}
v^1_{n}(t) =
 	\begin{cases}	
				u_{n}(t) 		&\text{ if } \,  0 \leq t \leq t_{n} = T+1 , 
		\\[1mm]
		u_{n} \big( t +  t^1_{n} - t_{n} \big)   	&\text{ if }  \, t > t_{n} =T+1.
	\end{cases}
\end{equation*}
On the other hand, 	if $u_n(T+1) < 1 - \varepsilon$,   set 
\[
t_n   =   \inf  \big\{  t \geq T + 1 \mid  u_n(t) = 1 - \varepsilon \big\} 
\]
and define
\[
t^2_n   =   \sup  \big\{  t \geq T + 1 \mid  u_n(t) = 1 - \varepsilon \big\} .
\]
In this case, 	we have  
\[
T+1 < t_n \le t^2_n
\]
and we set 
\begin{equation*}
v^2_{n}(t) =
 	\begin{cases}	
				u_{n}(t) 		&\text{ if } \,  0 \leq t \leq t_{n} , 
		\\[1mm]
		u_{n} \big( t +  t^2_{n} - t_{n} \big)   	&\text{ if }  \, t > t_{n}.
	\end{cases}
\end{equation*}
The functions $v^1_n$ or $v^2_n$ clearly satisfy statement $(iii)$.  Moreover,  due to $(a_2)$ and the definition of $T$,  we have
\begin{align*}
\calL_a^A(v^i_n)  & =	\int_{0}^{t_n}  \left( 1 - \sqrt{1 - |u_n^{\prime}|^2}  +  a(t)  W(u_n) \right)  \, \mathrm{d}t  
\\
		&\hspace{35mm} + \int_{t_n}^{+ \infty}   \left( 1 - \sqrt{1 - |u_{n}^{\prime} ( t +  t^i_{n} - t_{n} ) |^2}  +  a(t)  W \Big( u_n ( t +  t^i_{n} - t_{n} )  \Big) \right) \, \mathrm{d}t 
\\
&=
	\int_{0}^{t_n}  \left( 1 - \sqrt{1 - |u_n^{\prime}|^2}   +  a(t)  W(u_n) \right)  \,\mathrm{d}t  
\\
		&\hspace{35mm} 
					+  \int_{t^i_n}^{+ \infty}  \left( 1 - \sqrt{1 - |u_n^{\prime}(s)|^2}  +  a \Big( s-( t^i_{n} - t_{n} ) \Big)  W \big( u_n (s)  \big) \right)  \, \mathrm{d}s 
\\
  & \le	\int_{0}^{t_n}   \left( 1 - \sqrt{1 - |u_n^{\prime}|^2}  +  a(t)  W(u_n) \right)  \, \mathrm{d}t 
					+  \int_{t^i_n}^{+ \infty}  \left( 1 - \sqrt{1 - |u_n^{\prime}(s)|^2}  +  a(s)  W \big( u_n (s)  \big) \right)  \, \mathrm{d}s 
\\
  & \le	\calL_a^A(u_n).
\end{align*}
Therefore, replacing $u_n$ by $v^i_n$, if necessary, we may assume that the minimizing sequence ${(u_n)}_n$ satisfies (iii).

We next show that the sequence ${(t_n)}_n$ is bounded. Assume that $t_n > T+1$ for some $n\in\bbN_0$, otherwise the proof is complete. Since  ${(u_n)}_n$ is a minimizing sequence,  there is some constant $c>0$  such that
\[
\int_{0}^{+\infty} a(t) W(u_n) \, \mathrm{d}t    =		\int_{0}^{T + 1} a(t) W(u_n) \, \mathrm{d}t + \int_{T + 1}^{t_n} a(t) W(u_n) \, \mathrm{d}t + \int_{t_n}^{\infty} a(t) W(u_n) \, \mathrm{d}t   \leq c,
\]
for all $n\in\bbN_0$. By ($W_2^\prime$), we have
\begin{equation}
			\label{aux-abs-int}
\abs{ \int_0^{T+ 1}  a(t) W(u) \, \mathrm{d}t }	\le	(T + 1)\, \max_{[0,T+1]}  \abs{a} \, \max_{[-1,1]} W,
\end{equation}
for all $u \in \calX_A$,  while 
the third integral in the expression above is nonnegative by ($W_2^\prime$)  and the definition of $T$.
Then we obtain
\begin{equation}
			\label{aux-above}
\int_{T+ 1}^{t_n} a(t) W(u_n) \, \mathrm{d}t   \leq   c + (T + 1)\, \max_{[0,T+1]}\abs{a} \, \max_{[-1,1]}  W   =   c_1.
\end{equation}
On the other hand,  as $a$ is nondecreasing, we have, for all $n\in\bbN_0$,
\[
\inf_{[T+1,t_n]} a  =  a(T+1)  > 0,
\]
 by the definition of $T$. In addition, since $0 \le u_n(t) < 1-\eps $ for $T+1 \leq t  <  t_n$, we obtain
\begin{equation}
			\label{aux-below}
\int_{T+ 1}^{t_n} a(t) W(u_n) \, \mathrm{d}t 	\ge	\big(t_n - (T+1) \big)  \,  a(T+1)  \Big( \min_{[0,1-\eps ]} W \Big)  > 0 .
\end{equation}
Our claim  follows from estimates \eqref{aux-above} and \eqref{aux-below}.

\medskip
\noindent	\textbf{Step 2. Convergence to a minimizer.}									
From   \eqref{quadraticEstimates} and \eqref{aux-abs-int},  we have
\[
\frac{1}{2} \norm{u_n^\prime}_{L^2(\R)}   \le	\int_0^{+\infty} \left( 1 - \sqrt{1 - |u_n^{\prime}|^2}  \right) \mathrm{d}t	 \le	\calL_a^A (u_n)  -    \int_0^{T+ 1}  a(t) W(u) \, \mathrm{d}t    \le  c_1.
\]

Then arguing  as in Proposition \ref{prop:existence1Dmin},  we see that there is $ u \in W^{1,\infty} \big( [0,+\infty) \big)$ such that 
$
u_n$ converges uniformly to $u$ on every compact  $K\subset [0,+\infty)$, 
$u_n^\prime$ converges weakly to $u^\prime$  in $L^2 \big( [0,+\infty) \big)$
and 
$t_n  \longrightarrow   \overline{t}  \ge T+1$. 
Combining the weak lower semicontinuity of the kinetic part with Lebesgue's dominated convergence Theorem in $[0,T]$ and Fatou's Lemma in $[T,+\infty)$ for the potential part, 
we see that 
\begin{align*}
\mathcal{L}_{a}^A(u) 	&\le	\liminf_{n \to +\infty}   \int_0^{+\infty}  \left( 1 - \sqrt{1 - |u_n^{\prime}|^2}  \right) \mathrm{d}t   +	\lim_{n \to +\infty}    \int_0^{T}   a(t) W(u_n) \, \mathrm{d}t    +	\liminf_{n \to +\infty}   \int_T^{+\infty}   a(t) W(u_n) \,  \mathrm{d}t
\\
&\le  \liminf_{n \to +\infty}   \int_0^{+\infty}  \left(  1 - \sqrt{1 - |u_n^{\prime}|^2}  \, \mathrm{d}t   +  a(t) W(u_n)  \right) \, \mathrm{d}t   
\\
&=     \lim_{n \to +\infty} \calL_a^A(u_n) =  \inf_{\mathcal{X}_{A}} \mathcal{L}_{a}^A .
\end{align*}
Furthermore,  the uniform convergence on every compact  implies that 	
\[
u(0)=0,      \quad	   1-\eps \le u(t)  \le  1   \ \text{ for all }  t \ge  \overline{t}  \quad \text{ and }	\quad	\norm{ u^{\prime} }_{L^\infty(\R)} \leq 1.
\]
Finally, since 	
$\displaystyle \int_{T+1}^{+\infty}	W(u) \, dt$ is bounded,  we have  
$
\displaystyle	\limsup_{t\to +\infty} u(t) = +1
$
and we conclude from Lemma~\ref{lemma:limits}  that	
$\displaystyle \lim_{t \to + \infty} u(t) = 1$.

Therefore, $u \in \mathcal{X}_A$ and $\displaystyle \mathcal{L}_a^A(u) = \inf_{\mathcal{X}_A}\mathcal{L}_a^A$.

\medskip
\noindent	\textbf{Step 3.  The Euler-Lagrange equation.}
Proceeding in the same way as in Lemma \ref{lemma:solutionC1}, we obtain that $u$ is a $C^1([0,+\infty)) \cap W^{2,2}_{loc}([0,+\infty))$ solution of the boundary value problem
\begin{gather*}
	\bigg(  \frac{u'}{\sqrt{1 - |u'|^2}}  \bigg)'   =   a(t) W^{\prime}(u) \quad\text{ in } [0,+\infty),
\\
 u(0) = 0 , 	   \quad \lim_{t \to + \infty}	u(t) = 1,
\end{gather*}
with $\|u'\|_{L^\infty_{\text{loc}}(\R^+)} < 1$. 
Observe that $u$ is increasing on $[T,+\infty)$ since $a$ is nondecreasing. Otherwise there are  points $T \leq t_1 < t_2$  such that $u(t_1)=u(t_2)<1$ and the function $w \in \calX_A$ defined by 
\begin{equation*}
w(t) =
 	\begin{cases}	
				u(t) 		&\text{ if } \,  0 \leq t \leq t_{1}, 
		\\[1mm]
		u \big( t +  t_2 - t_{1} \big)   	&\text{ if }  \, t > t_{1}.
	\end{cases}
\end{equation*}
has a strictly lower action than $u$. Therefore $\lim_{t \to + \infty} u'(t)$ exists and is equal to zero. From this and the local estimates, we see that $\| u'\|_{L^\infty(\R)} < 1$.

\end{proof}

\bigskip

\bibliographystyle{plain} 
\bibliography{bibliography.bib}

\end{document}